\documentclass{article}
\pdfoutput=1
\usepackage{amssymb}
\usepackage{amsmath}
\usepackage{amsthm}
\usepackage{verbatim}
\usepackage{color}
\usepackage{url}
 \usepackage[all]{xy}
\usepackage{fancyhdr}
\usepackage{hyperref}
\usepackage{mathtools}

\newtheorem{thm}{Theorem}
\newtheorem{prop}{Proposition}
\newtheorem{lem}{Lemma}
\newtheorem{cor}{Corollary}

\newtheorem{rmk}[thm]{Remark}

\newcommand{\sabs}[1]{\left|#1\right|}
\newcommand{\sparen}[1]{\left(#1\right)}
\newcommand{\norm}[1]{\sabs{\sabs{#1}}}
\newcommand{\R}{\mathbb{R}}

\newcommand{\der}{\mathrm{d}}
\newcommand{\abs}[1]{\left\lvert #1 \right\rvert}
\newcommand{\aabs}[1]{\left\| #1 \right\|}

\newcommand{\rt}{I}

\numberwithin{equation}{section}

\begin{document}

\title{Recovery of the sound speed for the Acoustic wave equation from phaseless measurements\\\vskip 0.8cm}
\author{Joonas Ilmavirta\thanks{Department of Mathematics and Statistics, University of Jyv\"askyl\"a, P.O. Box 35 (MaD), FI-40014 University of Jyv\"askyl\"a, Finland}  \and Alden Waters\thanks{Department of Mathematics, University College London, Gower Street, London, WC1E 6BT, United Kingdom}}
\date{}

\maketitle \vskip 0.5cm
\begin{abstract}

We recover the higher order terms for the acoustic wave equation from measurements of the modulus of the solution. The recovery of these coefficients is reduced to a question of stability for inverting a Hamiltonian flow transform, not the geodesic X-ray transform encountered in other inverse boundary problems like the determination of conformal factors. We obtain new stability results for the Hamiltonian flow transform, which allow recovery of the higher order terms. Previous techniques do not measure the full amplitude of the outgoing scattered wave, which is the main novelty in our approach. 
\end{abstract}


\section{Introduction}

Scattering is a general physical process where some forms of radiation, such as light or sound, or moving particles are forced to deviate from a geodesic trajectory by a path due to localized non-uniformities in the medium through which they pass. Conventionally, this also includes the deviation of reflected radiation from the angle predicted by the law of reflection \cite{stover}. Scattering may also refer to particle-particle collisions between molecules, atoms, electrons, photons and other particles \cite{kolton}. The types of non-uniformities which can cause scattering are sometimes known as scatterers.  These include particles and surface roughness. The type of stability estimate we prove on the phaseless measurements of the solution insure that the energy of the waves uniquely determines the scatterer. 

The theory of signal processing and inverse problems has seen a recent increase in a class of so-called phaseless measurements. Often in experiments, when a source wave is measured, the only part of the information available to experimenters is the modulus of the wave from the source.  In signal processing, algorithms found in~\cite{demanet} and~\cite{candes} are focused on the recovery of waves from a sequence of Fourier modes. We are interested in phaseless measurements to recover scattering terms for the acoustic wave equation and generalized Helmholtz equation. The modulus of the wave corresponds to the energy density of the wave at a given point. The inverse boundary value problem differs from the questions examined in signal processing where often times one is dealing with incomplete data sets. In particular, we show that an idea of which partial differential equation the wave comes from is enough to give a full reconstruction of the coefficients modulo diffeomorphism. The scattered wave then completely determines the scatterer.  These results are supported by the numerical work in~\cite{bellet} and are applicable to other operators which admit a Gaussian beam type solution. The problem differs from the both of the author's previous work~\cite{aw,jo} because the terms which we are recovering come from higher order terms which control the bicharacteristic flow associated to the Hamiltonian governing the partial differential equation. 

Practical applications to phaseless problems are varied --- and one such application is multi-wave tomography. In multi-wave tomography usually some type of wave is sent to a portion of the body which is being imaged. In electromagnetic or optical radiation tomography the wave interaction with the tissues of the patient are measured~\cite{ammaribio}. Naturally one cannot measure inside the patient, so some initial boundary value problem must be considered. Similarly, to image the Earth, one has to send waves of some kind through the planet and make measurements at the surface. One such mathematical model of the emitted ultrasound waves is the acoustic wave equation with a high-frequency source term. 

Let $M\subset\R^d$ be a bounded and smooth manifold. Let $g$ be a Riemannian metric on $\R^d$ which agrees with the Euclidean one outside $M$. Let the standard Laplace-Beltrami operator be denoted as
\begin{align}
\Delta_g=\frac{1}{\sqrt{\det g(x)}}\frac{\partial}{\partial x^k}\sparen{g^{ki}(x)\sqrt{\det g(x)}\frac{\partial}{\partial x^i}}
\end{align}
in local coordinates with $g(x)=(g_{ik}(x))$, and $(g^{ik}(x))=(g_{ki}(x))^{-1}$. We consider manifolds, $M$, which are smooth ($C^{\infty})$. We write the local coordinates as $(x^1, . . ,x^d)$. We also assume the manifolds have a boundary.  

The generalized Helmholtz equation may be written as
\begin{align}\label{acoustic}
Lu=\Delta_gu+(i\alpha\lambda+\lambda^2) n^2(x)u=h(x,\lambda) \quad x\in \mathbb{R}^d.
\end{align}
The scalar $\lambda$ is large and $n^2(x)$ is sound speed. The number $\alpha$ helps describe the attenuation coefficient and $\alpha\in (0,1)$. It is zero in the limit of zero absorption. The source $h(x,\lambda)$, which emits the waves, we also assume to depend on $\lambda$ and be compactly supported in $x$ with codimension $1$. Sources will be modeled after Gaussian beams following \cite{olof}. We may pick sources anywhere inside the domain $M$, however we chose a particular set of them in order to provide a complete reconstruction of $n^2(x)$.  While this equation is known as the generalized Helmholtz equation, multiplication by a prefactor, $\exp(i\lambda t)$, takes the solution $u$ and turns it into a solution $\exp(i\lambda t)u$ to the acoustic wave equation when $n^2(x)\equiv 1$ on $\partial M$, and $\alpha\rightarrow 0$
\begin{align}\label{wave}
\partial_t^2u-n^{-2}(x)\Delta_gu=\exp(i\lambda t)h(x,\lambda) \quad x\in \mathbb{R}^d.
\end{align}
The acoustic wave equation models the scattering of waves in the Earth's core. 

We chose to model our solutions to the wave equation with Gaussian beams. The existence of Gaussian beam solutions to the wave equation has been known since the 1960's first in connection with lasers, see Babic and Buldyrev \cite{bb}. They were also used in the analysis of propagation of singularities in PDEs by Hormander \cite{hormander} and Ralston \cite{ralston77}. In the context of the Schr\"odinger equation first order beams correspond to the so-called classical coherent states. The higher order versions of these equations have been introduced to approximate the Schr\"odinger equation solutions in quantum chemistry by Heller \cite{heller}, Hagedorn \cite{hag}, and Herman and Kluk \cite{hk}. 

Given a smooth, strictly convex, bounded domain $M$ equipped with a metric $g$ we assume that on the boundary of $M$, $\partial M$ that $n^2(x)\equiv1$. The measurements we consider give then data of the form
\begin{align}
\{(x,\abs{u(x)}): x\in\partial M\}
\end{align} 
with the collection of $h^{x_0,\omega_0}(x,\lambda)$ varying over all $x_0\in\partial M$ and $\omega_0\in S_xM$, which are inward pointing into the manifold. We need this collection of data in order to give a complete reconstruction of $n^2(x)$. We consider the metric $g$ to be fixed and $n^2(x)$ to vary. This collection of measurements is a `true' phaseless problem in contrast to phaseless backscattering measurements which were recently investigated in~\cite{kbackscatter}. We see that the measurements for both equations \eqref{wave} and \eqref{acoustic} coincide. Thus we have developed a robust model for the scattering of both waves traveling through the earth's core \eqref{wave} and quantum mechanical particles in \eqref{acoustic}.  

In this paper, we derive a stability result for the higher order coefficients of the acoustic wave equation~\eqref{acoustic} for fully phaseless measurements. Stable reconstructions have been made from Robin conditions for lower order terms than considered here~\cite{ammari}. In the related case, for the generalized Helmholtz equation for Dirichlet boundary conditions in~\cite{isakovacoustic,increasing} and Robin conditions in~\cite{kexperiment,khyperbolic} the potential can also be recovered. However stability estimates from phaseless measurements have not been previously given.  In~\cite{kunique,kphaseless}, uniqueness results in dimension 3 for lower order terms than the ones considered are derived from phaseless measurements.  These papers are predicated on analyticity arguments, which require data in a small neighborhood of the source. The question of phaseless stability from internal measurements for Schr\"odinger was also examined in~\cite{al2}.  These measurements are in contrast to the boundary data we require.  We are not able to prove uniqueness results unless $\lambda\rightarrow \infty$. However, physically this is equivalent to setting the $n^2(x)$ term equal to zero. The inverse problem of recovering the source for wave equations from Dirichlet boundary conditions is also examined in~\cite{source,source2,homan,rakesh,multiwave}. The major difference is that we are able to recover higher order terms which make the bicharacterstic flow different.

\section{Statement of the Main Theorem}

Our main theorem can be stated as follows. We consider sources $h^{x_0,\omega_0}(x,\lambda)$. with support in $\mathbb{R}^{d-1}$. Here we view $M$ as an embedded submanifold of $M'$, a larger manifold. Sources are indexed by $x_0$ and $\omega_0$. We let $\nu(x)$ denote the outward unit normal to the boundary at $x$. Let the set
\begin{align}\label{inward}
\partial_+\mathcal{S}M=\{ (x_0,\omega_0): x_0\in\partial M,\,\, \langle\nu(x_0),\omega_0\rangle>0,\,\, \omega_0\in S_xM\},
\end{align}
be the range of $x_0,\omega_0$ where $S_xM$ denotes the unit sphere bundle of the manifold at $x$. 
The sources are defined for $x\in \mathbb{R}^{d-1}$ as:
\begin{align}\label{source}
h^{x_0,\omega_0}(x,\lambda)=2i\lambda\chi_\lambda(x)\left( \omega_0\cdot \nu+\mathcal{M}(0)(x-x_0)\cdot \nu+\nabla_g\chi_\lambda\cdot\nu \right)\exp(i\lambda\tilde{\psi}(x))
\end{align}
with 
\begin{align}
\tilde{\psi}(x)=(x-x_0)\cdot \omega_0+\frac{1}{2}(x-x_0)\cdot \mathcal{M}(0)(x-x_0).
\end{align}
Here $\chi_\lambda(x)$ is a smooth compactly supported function of codimension 1 in a neighborhood of the set
\begin{align}
\{x\in M', |x-x_0|<\lambda^{-1/2d}\}
\end{align}
 which is perpendicular to $\omega_0$. This set also contains the support of $\tilde{\psi}(x)$. The flow for the ray path of $H$ is defined by the set of ODEs in local coordinates on the manifold: 
\begin{align}\label{flow}
\frac{dx_i}{ds}=2g^{ij}(x(s))p_j, \quad \frac{dp_i}{ds}=-\frac{\partial g^{jk}(x(s))p_kp_j}{\partial x_i}+\frac{\partial n^2(x(s))}{\partial x_i}
\end{align}
 with initial data $(x_0,\omega_0)$. The matrix $\mathcal{M}(0)$ is a complex matrix satisfying
\begin{align}\label{ice}
\mathcal{M}(0)=\mathcal{M}(0)^T, \qquad \mathcal{M}(0)x(0)=\dot{p}(0), \quad \Im \mathcal{M}(0)\,\, \textrm {positive definite on} \,\, \dot{x}(0)^{\perp}.
\end{align}

We know there is a solution $u^{x_0,\omega_0}$ to the following equation for each $x_0,\omega_0$:  
\begin{align}
(\Delta_g+(\lambda^2+i\alpha\lambda)n^2(x))u^{x_0,\omega_0}=h^{x_0,\omega_0}(x,\lambda),
\end{align}
We let $\epsilon_0\in (0,1)$ and $\tilde{n}^2=n_1^2-n_2^2$. We have the following stability result on $\tilde{n}^2$ where the measurements are the amplitudes of the collection of $u^{x_0,\omega_0}$ corresponding to different $n^2$, ranging over sources indexed by $\partial_+\mathcal{S}M$. Let diam$_H(M)$ denote the maximal radius of the manifold with respect to the flow $H$. We let $g'$ denote the extended metric which is $g$ on $M$ and Euclidean on the exterior. We impose the condition the manifold $M'$ contains $M$ and be such that $M'$ is simple with respect to the metric $n^{-2}g'$. We recall that a simple manifold is one which is strictly geodesically convex with respect to the metric $g$, and has no conjugate points. 

\begin{thm}\label{main} Let $N>\max \{(1-d)/2+2+2s,0\}+1$, $s>d/2$. Then there exists a constant $C_1$, which depends on $\mathrm{diam}_H(M)$, the $C^{2}(M)$ norm of $\tilde{n}^2$, and a constant $C_2$, which depends on $\mathrm{diam}_H(M)$ and the $C^{N+s}(M)$ norm of $n_i^2(x), i=1,2$, such that if $u_1^{x_0,\omega_0}$ and $u_2^{x_0,\omega_0}$ solve the radiation problem with  coefficients $n_1^2$ and $n_2^2$ then it follows that if
\begin{align}\label{eq:smallness}
&
\norm{n_1^2-n_2^2}_{C^3(M)}<\epsilon;\,\,\, \lambda^{-1}<\epsilon_0, \qquad  \delta=\sup\limits_{x\in \partial M; (x_0,\omega_0)\in\partial_+\mathcal{S}M}||u_1^{x_0,\omega_0}|-|u_2^{x_0,\omega_0}||<\epsilon_0,
\end{align}
then this implies
\begin{align}
\norm{n_2^2-n_1^2}_{L^2(M)}\leq C_1\sparen{\frac{C_2}{\lambda^{\beta'}}+\delta}^{\frac{1}{2}}.
\end{align}
for some $\beta'\in (0,1)$ depending on dimension $d$, with $\lambda\epsilon>1$. 
\end{thm}
The uniqueness corollary follows immediately.
\begin{cor}
Assume $\delta=0$ and that the assumptions of Theorem~\ref{main} are satisfied for all large $\lambda$. Then $n_1^2=n_2^2$.  
\end{cor}
Our results indicate that there is very little stability to be expected from such a problem. 

\begin{rmk}
The standard problem see for example~\cite{LO,M} is to consider 
\begin{align}\label{standard}
&(-n^{-2}\Delta_g-\lambda^2)u=0, \\&
u|_{\partial M}=\tilde{h}(x,\lambda). \nonumber
\end{align}
with $\tilde{h}(x,\lambda)$ a compactly supported function on $\partial M$. Here one would recover the factor $n^{-2}$ from the Dirichlet-to-Neumman map with $\tilde{h}(x,\lambda)$ a high frequency source, with measurements of the form 
\begin{align}
 \int\limits_{\partial M}|\partial_{\nu}(u_1-u_2)|^2\,d_gS
\end{align}
We see that
\begin{align}\label{ineq}
\int\limits_{\partial M}||u_1|^2-|u_2|^2|\,d_gS \leq \int\limits_{\partial M}|u_1-u_2|^2\,d_gS
\end{align}
with $u_1, u_2$ corresponding to different $n_i$. Here $d_gS$ denotes the surface measure on $M$. Our problem is more physical because it corresponds to measuring the amplitudes of two different waves as on the left hand side of \eqref{ineq}, rather than super-imposing them and measuring their difference. We are also able to include an attenuation coefficient which is the $\alpha\rightarrow 0$ limit of \eqref{standard}. 
\end{rmk}

Generically, we are looking for an asymptotic model to~\eqref{acoustic} of the form
\begin{align}
U(x)=A(x,\lambda)\exp(i\lambda\psi(x))=\sum\limits_j\frac{a_j(x)}{\lambda^j}\exp(i\lambda\psi(x)),
\end{align}
which we show in the high frequency limit solves the equation~\eqref{acoustic} up to suitable error terms. The variable $x\in M$ and $s=s(x)$ is a parameter which helps describe the bicharacterstic flow in terms of the coordinates on $M$. We use a Gaussian beam Ansatz which involves the construction of a phase function $\psi(x)$ and an amplitude $a$ for the zeroth order beams in local coordinates as
\begin{align}\label{SMA}
&\psi(x)=S(s)+(x-x(s))\cdot p(s)+\frac{1}{2}(x-x(s))\cdot \mathcal{M}(s)(x-x(s))\\&
a(s,x)=a_0(s)+\mathcal{O}(d_{g'}(x,x(s)))\nonumber
\end{align}
where $x(s)$ is a curve which describes a Hamiltonian flow, given by \eqref{flow} below, and $S(s), \mathcal{M}(s)$ will be specified functions of $s$ determined by the Gaussian beam Ansatz. For the construction of the Ansatz, we follow the work of~\cite{olof} quite closely. In the previous work we considered the operator
\begin{align}
\tilde{L}=\Delta_{\R^d}+\lambda^2+i\lambda \alpha n^2(x),
\end{align}
while here we use 
\begin{align}
L=\Delta_g+\lambda^2n^2(x)+i\lambda \alpha n^2(x).
\end{align}
This corresponds to the acoustic wave equation. These operators have corresponding Hamiltonian flows given by the Hamiltonian functions which are $\tilde{H}=|p|^2-1$ and $H=|p|_g^2-n^2$. The main problem of isolating the X-ray transform of the coefficients $n^2(x)$ is to control the null-bicharacteristics. In local coordinates, the ordinary differential equations which govern the ray path $\{(s,x(s)):0\leq s\leq T\}$ along which solutions are concentrated are given by \eqref{flow} which is associated to $L$, resp. 
\begin{align}
\frac{d^2x(s)}{ds^2}=0,
\end{align}
which is associated to $\tilde{L}$. The second equation gives that the ray paths in $\mathbb{R}^d$ are straight lines while the first one does not. The problem of recovering $n^2(x)$ for the operator $L$ is more difficult, and we address it in this paper. \\

In the Calder\'on Problem in Conformally Transversal Geometries,~\cite{ks} and also~\cite{DKLS} reduces the question of boundary distance rigidity to a question of invertibility of the geodesic X-ray transform. As a consequence of their work, they reduce the question of recovery of source terms for several operators from the Dirichlet-to-Neumann maps to a question of invertibility of the geodesic X-ray transform. In this paper, we choose phaseless data as our measurements and show that the question of recovery of sound speed amounts to a question of invertibility and stability of a so-called flow transform. The results presented here have applications to other operators for which the Hamiltonian flow and the geodesic flow do not coincide. In order to prove our results, we introduce a condition which one can think of as a generalization of the condition of Bardos-Lebeau-Rauch~\cite{BLR} for Hamiltonian flows. Specifically, we require the Hamiltonian flow to be simple. 

\emph{Notation:} For two functions $f,g$, we write $f \sim g$ if there exists a constant $C>0$ such that $C^{-1} f \leq g \leq Cg$. We denote $d_g(x,y)$ the distance between the points $x,y$ defined by the Riemannian metric $g$. 

\section{Construction of Solutions}

We make more precise the explicit nature of the solutions. We will use the Ansatz, \begin{align}
U(x)=\sum\limits_{j=0}^l\exp(i\lambda\psi(x))a_j(x)\lambda^{-j}
\end{align}
to build asymptotic solutions to~\eqref{acoustic} in the high frequency limit, and then give estimates on the difference between these approximate solutions and the true solutions. We will follow~\cite{olof} for the
Gaussian beam Ansatz in constructing the phase $a(x)$ and amplitude $\psi(x)$. 
We claim:
\begin{thm}
\label{lma:gaussian}
There is a zeroth order Gaussian beam which solves the Helmholtz equation~\eqref{acoustic} in the free space with no source term. In local coordinates it takes the form
\begin{align}\label{approximation}
U(x)=(a_0(s)+\mathcal{O}(d_g(x,x(s)))\exp(i\lambda\psi(x))\phi(x)+\mathcal{O}(\lambda^{-1}),
\end{align}
where $x(s)$ is a curve in space defined by \eqref{flow} with initial condition $(x_0,\omega_0)\in\partial_+\mathcal{SM}$ and $\phi(x)$ is a cutoff which will be defined in a neighborhood of the curve
\begin{align}
a_0(s)=\chi_{\lambda}(x)\exp\sparen{\int\limits_0^s -\alpha n^2(x(t))-\mathrm{tr}\mathcal{M}(t)\,dt}.
\end{align}
and $\psi(x)$ is given by \eqref{SMA} and $\mathcal{M}(s)$ is a positive definite complex matrix whose evolution is governed by a Ricatti equation.
\end{thm}
The $N^{th}$ order Gaussian beam Ansatz can be constructed accordingly. However, for the purposes of this paper the zeroth order terms are the most important. 

\begin{rmk}
One might be tempted to use a real-phase Ansatz which should be of the form
\begin{align}
\exp(i\lambda\tilde{\psi}(x))\sparen{1+\frac{A(x)}{|x-x_0|}},
\end{align}
with $s$ a parameter describing the bicharacteristic flow, and $\tilde{\psi}(x)$ a purely real function, $A(x)$ encodes information about the flow transform of $n^2(x)$ or a similar solution found in~\cite{isakovbook}.  However, it is not possible to obtain information about the coefficient without the use of a weighted Sobolev space and this is more difficult since this clearly introduces singularities in the denominator.  The Gaussian beam Ansatz here was chosen because of the good control over the error estimates in higher Sobolev norms, which allows for the use of the embedding theorems in the section on observability inequalities. Also, because of the presence of the tail, one does not need to construct solutions which are localized along a central ray. This is necessary for the estimates~\eqref{hi} and~\eqref{mainsup} to hold, since Green's theorem is no longer available. However, the real phase construction has been efficient in recovery of coefficients for the Dirichlet-to-Neumann problem see~\cite{isakovbook,isakovdn}.
\end{rmk}
We need the following lemma to start off with:
\begin{lem}\label{lem:time2}
Every geometric optics solution concentrates on an open set around the ray path $\{(s,x(s)):0\leq s\leq T\}$ for some finite $T$, with $T>0$. The flow for the ray path of $H$ is defined by the set of ODEs in local coordinates on the manifold: 
\begin{align}
\frac{dx_i}{ds}=2g^{ij}(x(s))p_j, \quad \frac{dp_i}{ds}=-\frac{\partial g^{jk}(x(s))p_kp_j}{\partial x_i}+\frac{\partial n^2(x(s))}{\partial x_i}
\end{align}
and also
\begin{align}\label{time}
\frac{d}{ds}\psi(x(s))=\nabla_g\psi(x(s))\cdot \nabla_pH(x(s),p(s))=n^2(x(s))
\end{align}
\end{lem}
\begin{proof}
We have $H(x,p)=|p|_g^2-n^2$ associated to $H(x,\nabla_g \psi(x))$, so we set $p=\nabla_g\psi(x)$. We are looking for a solution to $H(x,\nabla_g \psi(x))=0$ because we want to solve \eqref{eikonal} to high order. We differentiate $H(x,\nabla_g\psi(x))$ with respect to $x$. This gives the following relationship:
\begin{align}\label{back}
\nabla_xH(x,\nabla_g\psi(x))+D^2\psi(x)\nabla_pH(x,\nabla_g\psi(x))=0
\end{align}
where $D^2$ represents the Hessian associated to $\psi$. For any curve $y(s)$ we have the following identity
\begin{align}\label{bye1}
&\frac{d}{ds}\nabla\psi(y(s))=D^2(\psi(y(s))\frac{dy(s)}{ds}=\\& \nonumber
D^2(\psi(y(s))\left(\frac{dy(s)}{ds}-\nabla_pH(y(s),\nabla_g\psi(y(s))\right)-\nabla_xH(y,\nabla_g\psi(y(s))
\end{align}
by substitution of \eqref{back}. It follows that if we are seeking a null-bicharacteristic curve that 
\begin{align}
&\frac{dx(s)}{ds}=\nabla_pH((x(s)),\nabla_g\psi(x(s)))\\& \nonumber
\frac{dp(s)}{ds}=\frac{d}{ds}(\nabla_g\psi(x(s))=-\nabla_xH(x(s),\nabla_g\psi(x(s)))
\end{align}
so that \eqref{bye1} vanishes. Moreover we see for our  particular $H$ that
\begin{align}
\frac{d}{ds}\psi(x(s))=\nabla_g\psi(x(s))\cdot \nabla_pH(x(s),p(s))=n^2(x(s)).
\end{align}
\end{proof}

\begin{proof}[Proof of Theorem~\ref{lma:gaussian}]
We have that
\begin{align}
LU=\sum\limits_{j=-2}^l\exp(i\lambda\psi(x))c_j(x)\lambda^{-j}.
\end{align}
The coefficients $c_j$, $j=0,1,\dots,l$, are defined recursively as follows
\begin{equation}
\label{eikonal}
\begin{split}
c_{-2}&=(n^2(x)-|\nabla_g \psi|^2)a_0(x)\equiv E(x)a_0,  \\
c_{-1}&=i\alpha n^2(x)a_0+\nabla_g\cdot(a_0\nabla_g\psi)+\nabla_ga_0\cdot\nabla_g\psi+E(x)a_1,\\
c_j&=i \alpha n^2(x)a_{j+1}+\nabla_g\cdot(a_{j+1}\nabla_g\psi)+\nabla_ga_{j+1}\cdot \nabla_g\psi+E(x)a_{j+1}+\Delta_ga_j.
\end{split}
\end{equation}
If we Taylor expand the coefficients $a_j(x)$ around the central ray $x(s)$, we are looking for a system of differential equations for the coefficients. If we have chosen $c_{-2}$ and $c_{-1}$ to vanish on the ray to third and first order respectively and define $S$ as in~\ref{SMA}. This leads to the following set of differential equations: 
\begin{align}
d_sS(s)=2n^2(x(s)), \quad d_sa_0(s)=-\mathrm{tr}(\mathcal{M}(s))a_0-\alpha n^2(x(s))a_0.
\end{align}

Denoting $\mathcal{M}_{jk}(s)=\partial^2_{x_jx_k}\psi(x(s))$, we also must have
\begin{align}
-d_s\mathcal{M}_{jk}(s)=\sum_{\ell,r}\big(\partial^2_{p_\ell p_r}\tilde{H}(x(s),p(s))\big)\mathcal{M}_{j\ell}(s)\mathcal{M}_{kr}(s)\\
+\sum_{\ell}\Big(\partial^2_{x_jp_\ell}\tilde{H}(x(s),p(s))\mathcal{M}_{\ell k}(s)
+\partial^2_{x_kp_\ell}\tilde{H}(x(s),p(s))\mathcal{M}_{\ell j}(s)\Big)\\
+\partial^2_{x_jx_k}\tilde{H}(x(s),p(s))
\end{align}
with $\tilde{H}(x,p)=p\cdot\nabla_pH(x,p)$.

These equations correspond to differentiating the equation
\begin{align}\label{time3}
\partial_s\psi(x)+\nabla_g\psi(x)\cdot \tilde{H}(x,\nabla_g\psi)=0
\end{align}
with respect to $x$ and evaluating along the curve $x(s)$ with respect to the constraint $H(x,\nabla_g\psi)=0$. This equation is the correct one to consider by the derivation of \eqref{time}. These equations can be easily checked \cite{olof2}. The matrix $\mathcal{M}$ is known as the Hessian matrix. Higher order beams can be constructed by requiring $c_{-2}$ vanishing to higher order on $\gamma$. One can require the $c_j's$ with $j>-2$ also vanish to higher order and obtain a recursive set of linear equations for the partial derivatives of $a_0,a_1, . .. ,a_l$. More precisely, for an $N^{th}$ order beam $l=\lceil N/2 \rceil-1$, and $c_j(x)$ should vanish to order $N-2j-2$ whenever $-2\leq j\leq l-1$. We construct higher order beams, but we are only interested in the precise form of the zeroth order terms as they are sufficient for our asymptotics. 
We see that
\begin{align}\label{approximation1}
a_0(s)=\chi_{\lambda}(x)\exp\sparen{\int\limits_0^s -\alpha n^2(x(t))-\mathrm{tr}\mathcal{M}(t)\,dt}.
\end{align}
The phase $\psi$ needs to verify the conditions
\begin{align}
\psi(x(s))=S(s), \quad \nabla\psi(x(s))=p(s), \quad D^2\psi(x(s))=\mathcal{M}(s),
\end{align}
compatible with~\ref{SMA}, and differentiation of \eqref{time} with respect to $x$. We use the initial data $S(0)=0$ and $\mathcal{M}(0)$ satisfies \eqref{ice} cf.~\cite[Section~2]{olof}. If the matrix $\mathcal{M}(s)$ satisfies \eqref{ice}, then for all $s$ the matrix $\mathcal{M}(s)$ inherits the properties of $\mathcal{M}(0)$;
\begin{align}
\mathcal{M}(s)\dot{x}(s)=\dot{p}(s) \qquad \mathcal{M}(s)=\mathcal{M}(s)^{\perp}
\end{align}
and $\Im\mathcal{M}(s)$ is positive definite on the orthogonal complement of $\dot{x}(s)$, cf~\cite{ralston77}. 
In order to write down such an $\psi(x)$ we need to be able to write $s$ as a function of the coordinate variables $x\in M$. We know $x(s)$ traces out a smooth curve $\tilde{\gamma}$ in $\mathbb{R}^d$ and if we assume $x(s)$ is non-trapping then this curve is a straight line when $s$ is sufficiently large. We consider $R$ large enough so that the set $\{x: |x|_{g'}<6R\}$ contains $\overline{M}$. We set
\begin{align}
\Omega(\epsilon_0)=\{x: |x|_{g'}\leq 6R \qquad d_{g'}(x,\tilde{\gamma})\leq \epsilon_0\}
\end{align}
as a tubular neighborhood of $\tilde{\gamma}$ with radius $\epsilon_0$ in the ball $\{|x|_{g'}\leq 6R\}$. Choosing $\epsilon_0$ sufficiently small we can uniquely define $s=s(x)$ for all $x\in \Omega(\epsilon_0)$ such that $x(s)$ is the closest point on $\tilde{\gamma}$ to $x$ provided $\tilde{\gamma}$ has no self-intersections. The variable $s$ is the analogue of the time variable for time dependent problems.

We now define a cutoff function $\phi_{\lambda}(x)\in C^{\infty}(\mathbb{R}^d)$ for $\lambda>0$ such that
\begin{align}\label{bye}
\phi_{\lambda}(x)=\left\{
\begin{array}{lr} 0 \quad \mathrm{if} \quad x\in \Omega^c(\lambda^{-1/2d}) \\
1 \quad \mathrm{if} \quad x\in \Omega(\lambda^{-1/2d})
\end{array}
\right.
\end{align}
One can arrange that there is a constant $C$ such that
\begin{align}
\sup_{x\in M}|\nabla_x^m\phi_{\lambda}|\leq C\lambda^{-m}.
\end{align}
We drop the subscript $\lambda$ for the rest of this paper. We used the fact that $\lambda^{-1}<\epsilon_0$ so the definition of the parameter $s$ makes sense as above. 

We also recall the following result:
\begin{lem}[{\cite[Corollary~5]{waterss}}]\label{expsize}
Let $\psi(x)$ be the phase function of a zeroth order beam. We have
\begin{align*}
\exp(-2\lambda\Im\psi(x))\sim\exp(-\lambda Cd_{g'}(x,x(s))^2),
\end{align*}
where $C$ is independent of $\lambda$. 
\end{lem}
Let $B$ denote the set 
\begin{align*}
B=\{x: d_{g'}(x,x(s))>\lambda^{-(\frac{1}{2}-\sigma)},\,\,\, 0\leq s\leq 6R\},  \qquad \sigma>0,\ \ \sigma \in\mathbb{R}.
\end{align*}
We conclude that since $2\Im\psi(s,x)\sim d_{g'}(x,x(s))^2$, $\exp(-2\lambda\Im\psi(x))$ is exponentially decreasing in $\lambda$ for all $x\in B$. 
Notice that we are taking more care to construct the cutoff functions than in~\cite{olof}, as they are crucial for the phaseless measurements. 
\begin{proof}
We need only observe that $\mathcal{M}(s)$ is a bounded and positive definite matrix. From the form of the phase functions constructed the desired result follows. 
\end{proof}
The construction of the localized cutoff finishes the construction of $U_{\lambda}$. A standard argument gives that $U_{\lambda}$ extends across each local coordinate chart to cover the ray $x(s)$ iteratively, c.f.~\cite[Section~7]{ks}. 
\end{proof}

\section{Introduction of the Source Terms}

We now introduce source functions. We summarize the results from \cite{olof} to show we can build such sources. We claim:
\begin{thm}
There exists a Gaussian beam solution which solves the equation~\eqref{acoustic} and is found by solving the Dirichlet problem on one side of hyper-planes which contain a source point.  
\end{thm}

This argument is similar to ~\cite[Section~2.1]{olof} and is repeated for completeness.
We let $\rho$ be a function such that $|\nabla_g\rho|=1$ on the hypersurface $\Sigma=\{x: \rho(x)= 0\}$.
Let $x_0$ be a point in $\Sigma$ and we let $(x(s),p(s))$ be the solution path --- in other words the null-bicharacteristics with $(x(0),p(0))=(x_0,n(x_0)\nabla_g\rho(x_0))$.
The hypersurface $\Sigma$ is given by $s=\sigma(y)$ with $\sigma(0)=0$ and $\nabla\sigma(0)=0$ where $x=(s,y)$ and $y=(y_1,. . .,y_{d-1})$ is transversal.
We let the optics Ansatz $U(x)$ have initial data $(x(0),p(0))$, and be defined in this tubular neighborhood.
We let $U^+$ be the restriction of $U$ to $\{x: \rho(x)\geq 0\}$.
Because we need to have a source term which is a multiple of $\delta(\rho)$ , we also need a second, `outgoing' solution $U^-$, defined on $\{x: \rho(x)\leq 0\}$. It is then equal to $U^+$ on the hypersurface $\Sigma$.
For example, we can write the ingoing and outgoing optics solutions as
\begin{align}
 U^+=A^+(x,\lambda)\exp(i\lambda\psi^+(x)), \qquad U^-=A^-(x,\lambda)\exp(i\lambda\psi^-(x)),
\end{align}
where \cite{olof} have set $\psi^+=\psi^-$ and $A^+=A^-$ on $\Sigma$. The requirement that their Taylor series coincide on the boundary is equivalent to setting $\partial_y^\alpha|_{y=0} \psi^+(\sigma(y))=\partial_y^\alpha|_{y=0} \psi^-(\sigma(y))$. We extend $U^+$ to be $0$ on $\{x: \rho(x)<0 \}$ and $U^-$ to be $0$ on $\{x:\rho(x)>0 \}$. We define our geometric optics Ansatz solution $U$ to be $U=U^++U^-$. We set $A=A^+=A^-$ on $\Sigma$. In order to add the source term, we notice that 
\begin{align}\label{diff}
&LU=i\lambda\sparen{\sparen{\frac{\partial\psi^+}{\partial\nu}-\frac{\partial\psi^-}{\partial\nu}}A(x,\lambda)+\frac{\partial A^+}{\partial\nu}(x,\lambda)-\frac{\partial A^-}{\partial\nu}(x,\lambda)}\exp(i\lambda\phi^+)\delta(\rho))\\&+f_{gb}= \nonumber
g_0\delta(\rho)+f_{gb}, \nonumber
\end{align}
where $\nu(x)=\nabla \rho(x)$ is the unit normal to $\Sigma$. We consider the singular part $g_0\delta(\rho)=h(x,\lambda)$ to be the source term and $f_{gb}$ the error. Then we obtain
\begin{align}
f=\exp(i\lambda\psi^+)\sum\limits_{j=-2}^lc_j^+(x)\lambda^{-j}+\exp(i\lambda\psi^-)\sum\limits_{j=-2}^lc_j^-(x)\lambda^{-j},
\end{align}
where we extend $c_j^+$ to be zero for $\rho(x)<0$ and $c_j^-$ to be zero for $\rho(x)>0$. We know by construction that $c_{-2}^{\pm}=\mathcal{O}(d_g(x,x(s))^3)$ and $c_{-1}^{\pm}=\mathcal{O}(d_g(x,x(s)))$, respectively for $1^{st}$ order beams. 

By construction we see that $h(x,\lambda)$ as defined in \eqref{source} satisfies \eqref{diff}, with $\rho(x)$ a boundary defining function. The incoming and outgoing Gaussian beams are matched at the boundary of $M$ by construction. 

\begin{rmk}
 While this procedure may seem backwards and a bit ad hoc, it is useful for deriving good error estimates which are needed for the use of the embedding theorems. 
\end{rmk}

\section{Error Estimates}

We claim for $N^{th}$ order Gaussian beams:
\begin{lem}
We have the following estimate for the Gaussian beam error:
\begin{align}
&\norm{f_{gb}(x)}^2_{H^m(|x|_{g'}<R)}\leq C\lambda^{-N+2+(1-d)/2+2m}.
\end{align}
with $C$ independent of $\lambda$ depending on then $C^{N+m}(M')$ norm of both the initial data and $\tilde{n}^2(x)$. 
\end{lem}

\begin{proof}
We have that the $c_j(x)^{\pm}$ are bounded and 
\begin{align}
c_{j}^{\pm}(x)=\sum\limits_{|\beta|=N-2j-2}d_{\beta,j}^{\pm}d_{g'}(x,x(s))^{\beta}, \qquad j=-2, . . . ,l-1,
\end{align}
where the $d_{\beta,j}$ are bounded by Taylor's theorem. We obtain
\begin{align}
|c_j^{\pm}(x)|_{g'}\leq C_jd_{g'}(x,x(s))^{N-2j-2}|\phi(x)|_{g'},
\end{align}
with the $C_j$ uniformly bounded independent of $\lambda$. As $\mathcal{M}(s)$ is a matrix with positive definite imaginary part,
\begin{align}
\mathrm{Im}\psi^{\pm}(x)\geq Cd_{g'}(x,x(s))^2
\end{align}
for some positive constant $C$. From the elementary inequality~\cite{olof} for $a,b>0$,
\begin{align}
b^p\exp(-ab^2)\leq C_pa^{-p/2}\exp(-ab^2/2), \qquad C_p=\sparen{\frac{p}{e}}^{\frac{p}{2}},
\end{align}
with $p=N-2j-2$, $a=\lambda c$ and $b=d_{g'}(x,x(s))$, $x\in \Omega(\lambda^{-1/2})$, we obtain
\begin{align}
|f_{gb}(x)|_{g'}\leq \exp(-\lambda\Im\psi^{\pm}(x))\sum\limits_{j=-2}^l|c_j(x)|_{g'}\lambda^{-j}.
\end{align}
see \cite{LUT}. It follows that
\begin{align*}
&|f_{gb}(x)|_{g'}\leq \exp(-\lambda cd_{g'}(x,x(s))^2)\sum\limits_{j=-1}^lC_jd_{g'}(x,x(s))^{N-2j-2}\lambda^{-j}\leq \\& C\exp(-\lambda cd_{g'}(x,x(s))^2/2)\sum\limits_{j=-2}^l\lambda^{-N/2+j+1}\lambda^{-j}\leq C\exp(-\lambda cd_{g'}(x,x(s))^2/2)\lambda^{-N/2+1} 
\end{align*}
with $C$ a generic constant independent of $\lambda$ depending on then $C^{N+m}(M')$ norm of both the initial data and $\tilde{n}^2(x)$. If we differentiate $c_j^{\pm}(x)$ and $\exp(i\lambda\psi(x))$, then we have for the $N^{th}$ order beams
\begin{align}\label{comperror}
&\norm{f_{gb}(x)}^2_{H^m(|x|_{g'}<R)}\leq \\& \qquad \qquad \quad C\lambda^{-N+2+2m}\int\limits_{\mathclap{x\in \Omega(\lambda^{-1/2})}}\exp(-2\lambda Cd_g(x,x(s))^2)\,dx \leq C\lambda^{-N+2+(1-d)/2+2m},
\nonumber
\end{align}
as claimed since $\Omega(\lambda^{-1/2d})\subset \Omega(\lambda^{-1/2})$ and the integrand is positive. 
\end{proof}

\begin{rmk}
Notice that this error now includes a function which is localized in a neighborhood of $\mathcal{O}(1/\sqrt{\lambda})$, which is more precise than in~\cite{olof}. The reason for the precise $\phi_{\lambda}$ is twofold: We need a localized solution for~\eqref{mainsup} to hold, and we are also considering a bounded domain. 
\end{rmk}

\section{Extension of the Ansatz}
The extension of the Gaussian beam Ansatz to $M'$ will provide us with the means to make observations further away from the source and still identify source terms. In the last section we built an Ansatz which is a subset of a slightly extended version of $M$. We view the manifold $M$ as an embedded subset of $\mathbb{R}^d$, so we have the following inclusions: 
\begin{align}
\overline{M}\subset \overline{M'}\subset \{x: |x|_{g'}<6R\}.
\end{align}
for $R$ sufficiently large. We briefly review the results of~\cite{olof}. We set
\begin{align}
\tilde{f}_u=LU-g_0\delta(\rho).
\end{align}
We would like this function to be supported in $|x|_{g'}<6R$ and be $\mathcal{O}(\lambda^{-1})$. We let $G_{\lambda}(x)$ be the Green's function for the Helmholtz operator $L$.
In order to extend our approximate solution $U$, we introduce a smooth cutoff $\chi_a(x)$ such that
\begin{align*}
&\chi_a(x)=1  \quad \text{for}\ \ |x|_{g'}<(a-1)R,\\&
\chi_a(x)=0 \quad \text{for}\ \ |x|_{g'}>aR.
\end{align*}
Now we set
\begin{align}
\tilde{U}(x)=\chi_3(x)U(x)+\int\limits G_{\lambda}(x-y)\chi_5(y)L[(1-\chi_3(y))U(y)]\,dy.
\end{align}
In~\cite{olof}, they prove 
\begin{align}\label{error2}
\norm{U-\tilde{U}}_{H^m(|x|_{g'}<6R)}=\mathcal{O}(\lambda^{-n}).
\end{align}
 Therefore the size of $\tilde{U}$ and the extension depends on the size of $f_{gb}$. Furthermore, using ideas in Vainberg~\cite{vainberg}, they also prove
\begin{align}\label{error1}
\norm{U-u}_{H^m(|x|_{g'}<6R)}\leq C\lambda^{-1}\norm{f_{gb}}_{H^m(|x|_{g'}<R)}.
\end{align}
The triangle inequality allows us to conclude
\begin{align}\label{error3}
\norm{u-\tilde{U}}_{H^m(|x|_{g'}<6R)}\leq C\lambda^{-1}\norm{f_{gb}}_{H^m(|x|_{g'}<R)}.
\end{align} 
The proofs done in \cite{olof} are for when $(M,g)$ has the Euclidean metric, but can be easily extended as long as $M$ is simple with respect to $g$. 

In conclusion, we have that:
\begin{lem}
There exists an extension of the Gaussian beam solution of the problem~\eqref{acoustic} to the whole space such that
\begin{align}
\norm{u-\tilde{U}}_{H^m(M')}\leq C\lambda^{-1}\norm{f_{gb}}_{H^m(|x|_{g'}<R)}.
\end{align} 
\end{lem}
Here we know $\tilde{U}$ is supported in $|x|<6R$ and we know $u$ is in $H^m(\mathbb{R}^d)$ with $g'$ metric so that for $R$ sufficiently large, $\norm{u-\tilde{U}}_{H^m(M')}$ is essentially the same as $\norm{u-\tilde{U}}_{H^m(|x|_{g'}<6R)}$. 

\section{Integral Transforms Generated by Hamiltonian Flows}

We now consider the problem of recovering a function stably from its integrals over the integral curves of a Hamiltonian flow.
A related nonlinear problem was considered in~\cite{AZ:em-rigidity}, and generalizations of the geodesic X-ray transform problem have been studied before (see e.g.~\cite{AD:surface,FSU:general-x-ray,R:radon-cormack-type}). However, we are not aware of any discussion of the present problem in the literature.

Let $(M,g)$ be a compact Riemannian manifold with boundary and $q\colon M\to\R$ a smooth function. We let $\Omega$ denote a magnetic field which is a closed $2$-form. We consider the law of motion governed by Newton's equation:
\begin{align}\label{mp}
\nabla_{\dot{\gamma}}\dot{\gamma}=Y(\dot{\gamma})-\nabla q(\gamma)
\end{align}
We have $Y:TM\rightarrow TM$ is the Lorentz force which is associated to $\Omega$, which is the map uniquely determined by 
\begin{align}
\Omega_{x}(\xi,\eta)=\langle Y_x(\xi),\eta\rangle
\end{align}
$\forall x\in M, \xi,\eta\in T_xM$. The curve which satisfies \eqref{mp} is called an $\mathcal{MP}$ geodesic as per \cite{AZ:em-rigidity}. The $\mathcal{MP}$ stands for Maupertius principle.  The equation \eqref{mp} defines a flow $\phi_t$ on $TM$. which we call an $\mathcal{MP}$ flow. Whenever $\Omega=0$ the flow is the potential flow and this corresponds to the Hamiltonian flow, corresponding to the Hamiltonian function $H\colon T^*M\to\R$ given by $H(x,p)=\frac12g_{ij}(x)p^ip^j+q(x)$ in local coordinates. For our purposes it is sufficient to examine potential flows, but our results could extend to situations with a magnetic term included, in which case the Hamiltonian would be $H(x,p)=|p+\alpha|^2_g+q(x)$, $\Omega=d\alpha$ for $\Omega$ exact. 

We refer to the curves generated by the Hamiltonian flow as $H$-geodesics. For an $\mathcal{MP}$ flow, the energy $\mathcal{E}(x,\xi)=\frac{1}{2}|\xi|_g^2+q(x)$ is an integral of motion. The law of conservation of energy says for all $\mathcal{MP}$ geodesics the energy is constant along the geodesic. 

The $\mathcal{MP}$ flow depends on the choice of energy level.  Consider an $H$-geodesic $\gamma\colon [0,T]\to M$.
The value of $H(\gamma(t),\dot\gamma(t))$ is independent of $t$, and we write $H_0$ for this constant.
 We assume that $H_0>\max\limits_{x\in M}q$ so we can set $S_{H_0}M=\mathcal{E}^{-1}(H_0)$, the bundle with energy $H_0$ over $M$. It is then necessary for $H_0$ to be larger than $\max\limits_{x\in M}q(x)$. Otherwise we get as $x\in M$ any vector $\xi\in S_{H_0}M$ has non-positive length. We let $\nu(x)$ the inward normal to $\partial M$ at $x$. We let the set of vectors
\begin{align}
\partial_+S_{H_0}M=\{(x,\xi)\in S_{H_0}M: \langle \xi,\nu\rangle \geq 0\}
\end{align}
For $(x,\xi)\in S_{H_0}M$ let $\tau(x,\xi)$ be the time such that the $\mathcal{MP}$ geodesic with $\gamma(0)=x, \dot{\gamma}(0)=\xi$ reaches $\partial M$. The function $\tau(x,\xi): S_{H_0}M\rightarrow \mathbb{R}$ is smooth by Lemma A.5 of \cite{AZ:em-rigidity}. 

Let $\Lambda(\cdot,\cdot)$ denote the $2^{nd}$ fundamental form of $\partial M$. The boundary of $\partial M$ is strictly $\mathcal{MP}$ convex if
\begin{align}
\Lambda(x,\xi) > \langle Y_x(\xi),\nu(x)\rangle -d_xq(\nu(x))
\end{align}
for all $(x,\xi)\in \partial_+S_{H_0}M$. 

For all $x\in M$, the $\mathcal{MP}$ exponential map at $x$ is the partial map
\begin{align}
\exp_x^{\mathcal{MP}}: T_xM\rightarrow M
\end{align}
which is given by the following:
\begin{align}
\exp_x^{\mathcal{MP}}(t\xi)=\pi\circ\phi_t(\xi) \quad t\geq 0,\,\, (x,\xi)\in S_{H_0}M
\end{align}
For all $x\in M$, $\exp_x^{\mathcal{MP}}$ is a $C^1$ smooth partial map and $T_xM$ is $C^{\infty}$ smooth on $T_xM/\{0\}$. We let $\exp_x^{H}$ denote the exponential map associated to the potential Hamiltonian ($\Omega=0$).  In analogy to the usual exponential map $\exp_x\colon T_xM\to M$, we have the Hamiltonian exponential map $\exp^H_x\colon T_xM\to M$ for each $x\in M$, defined by the Hamiltonian flow. 

Since our manifold has a boundary, this is only a partial map.
We say that the manifold is simple at energy $H_0$ if the map $\exp^H_x\colon (\exp^H_x)^{-1}(M)\to M$ is a diffeomorphism for each $x\in M$, and also that $\partial M$ is strictly $\mathcal{MP}$ convex.  We remark that simplicity implies $H_0>\max\limits_{x\in M}q$. It is more stringent to assume a system to be simple than to assume it to be non-trapping.

We define the Hamiltonian flow transform of a function $f\colon M\to\R$ at energy $H_0$ for $(x,\xi)\in \partial_+S_{H_0}M$ as
\begin{align}
&\rt_{H_0}f(\gamma)
=\int_0^{\tau(x,\xi)}f(\gamma(t)) \der t,\\&
I_{H_0}f: C^{\infty}(M)\rightarrow C^{\infty}(\partial_+S_{H_0}M)
\end{align}
where $\gamma$ is an $H$-geodesic of energy $H_0$.
We denote the corresponding normal operator by $N_{H_0}=\rt_{H_0}^*\rt_{H_0}$, with the adjoint operator
\begin{align}
I^*_{H_0}f: C^{\infty}(\partial_+S_{H_0}M)\rightarrow C^{\infty}(M).
\end{align}

For a fixed energy level $H_0$ we let $\gamma(t)$ be an $\mathcal{MP}$ geodesic. We consider the change of variables
\begin{align}
s(t)=\int\limits_0^t2(H_0-q(\gamma(t))\,dt.
\end{align}
It follows that $s$ is the arc-length of $\sigma(s)=\gamma(t(s))$ under the metric $\tilde{g}=2(H_0-q)g$. A version of Maupertuis principle says that $\sigma(s)=\gamma(t(s))$ is a unit speed magnetic geodesic of the system $(\tilde{g},0)$. We quote these results from \cite{AZ:em-rigidity} without proof. 

\begin{thm}[Theorem 2.1 in \cite{AZ:em-rigidity}]
Let $(M,g,q)$ denote an $\mathcal{MP}$ potential system on $M$ and $H_0$ a constant such that $H_0>\max\limits_{x\in M} q(x)$. We suppose $\gamma(t)$ is an $MP$ geodesic of $H_0$, then $\sigma(s)=\gamma(t(s))$ is a unit speed geodesic of the system $(M,\tilde{g},0)$ on $M$. 
\end{thm}

and
\begin{prop}[Prop 2.2 in \cite{AZ:em-rigidity}]
The $\mathcal{MP}$ potential system $(M,g,q)$ on $M$ of energy $H_0$ is simple if and only if so is system $(M,\tilde{g},0)$. 
\end{prop}

We use these two results to prove the following:
\begin{lem}
\label{lma:hamilton-riemann}
Suppose $(M,g,q)$ is simple at energy $H_0$ w.r.t. the Hamiltonian $H$.
Define a new metric $\tilde g$ conformal to $g$ by $\tilde g=2(H_0-q)g$.
Then $(M,\tilde g)$ is a simple Riemannian manifold and a reparametrization turns $H$-geodesics of energy $H_0$ to unit speed geodesics w.r.t. $\tilde g$.

Moreover, if $\tilde\rt$ denotes the X-ray transform w.r.t. the metric $\tilde g$, then $\rt_{H_0}f(\gamma)=\tilde\rt[f/2(H_0-q)](\gamma\circ r)$, where $r$ is the reparametrization.
\end{lem}

\begin{proof}
The first part follows from~\cite[Proposition~2.1]{AZ:em-rigidity} and the second part is a straightforward calculation.
\end{proof}

Let us remark that the boundary $\partial M$ is strictly convex with respect to $\tilde g$ if and only if $\Lambda(v,v)>-\partial_\nu q(x)$ for all $(x,v)\in S^*_{H_0}M$ with $x\in\partial M$ and $v\perp\partial M$. Here $\partial_\nu$ is the inward normal derivative and $\Lambda(\cdot,\cdot)$ the second fundamental form.
This calculation can be found in~\cite[Lemma~A.4]{AZ:em-rigidity}.

\begin{thm}
\label{thm:inj-stab-scalar}
Suppose $\dim M\geq2$.
If the manifold is simple for $H_0$, then $\rt_{H_0}$ is injective on $L^2(M)$.
Moreover, we have the stability estimate
\begin{equation}
\aabs{f}_{L^2(M)}
\leq
C\aabs{N_{H_0}f}_{H^1(M')},
\end{equation}
where $M'\subset M$ is a slightly extended manifold and $C$ is a constant depending on the manifold, $q$ and $H_0$.
\end{thm}

\begin{proof}
By Lemma~\ref{lma:hamilton-riemann} it suffices to show that the X-ray transform on the manifold $(M,\tilde g)$ is injective and stable.
Injectivity~\cite[Theorem~7.1]{DKSU:anisotropic} and stability~\cite[Theorem~3]{SU:tensor-stability} are known for simple manifolds.
\end{proof}

In fact, it suffices to assume that the underlying manifold $(M,g)$ is simple if $H_0$ can be taken arbitrarily large.
The following corollary follows immediately from Theorem~\ref{thm:inj-stab-scalar} above and Lemma~\ref{lma:simple-open} below.

\begin{cor}
\label{cor:simple-inj}
Let $(M,g)$ be a simple manifold of dimension two or higher.
Then for sufficiently large $H_0$
\begin{equation}
\aabs{f}_{L^2(M)}
\leq
C\aabs{N_{H_0}f}_{H^1(M')},
\end{equation}
where $M'\subset M$ is a slightly extended manifold and $C$ is a constant depending on the manifold, $q$ and $H_0$.
\end{cor}
The corollary applies, in particular, to the closure of any strictly convex and bounded Euclidean domain.

\begin{lem}
\label{lma:simple-open}
If $(M,g)$ is simple, then the manifold $M$ is simple at energy $H_0$ for large enough $H_0$.
\end{lem}

\begin{proof}
It is well known that simplicity is an open condition: small perturbations of simple metrics are still simple. Therefore for $H_0$ large enough the metric $(1-q/H_0)g$ is simple.
Rescaling the metric does not alter simplicity, so also the metric $\tilde g=2(H_0-q)g$ is simple.
The claim then follows from Lemma~\ref{lma:hamilton-riemann}.
\end{proof}

\begin{rmk}
This remark may be of interest for integral geometers, although not strictly relevant for this paper.

In a similar way we can prove injectivity results for the Hamiltonian ray transform of tensor fields.
For a smooth tensor field $f$ of order $m$ the formula of lemma~\ref{lma:hamilton-riemann} is $\rt_{H_0}f(\gamma)=\tilde\rt[(2H_0-2q)^{m-1}f](\gamma\circ r)$.
If solenoidal injectivity is known for the geodesic X-ray transform of tensor fields on the manifold $(M,\tilde g)$, then we know that $(2H_0-2q)^{m-1}f=\sigma\tilde\der h$ for some tensor $h$ of order $m-1$ (here $\tilde\der g$ is the covariant derivative w.r.t. the metric $\tilde g$ and $\sigma$ is symmetrization).
On simple surfaces solenoidal injectivity is known for tensors of all orders~\cite[Theorem~1.1]{PSU:tensor-surface}, and there are several solenoidal injectivity results for higher dimensions as well (see eg.~\cite{PSU:tensor-survey,PSU:anosov-mfld,SU:tensor-stability,AR:one-form}).
In particular, solenoidal injectivity is known for $m=0$ and $m=1$ (see~\cite{AR:one-form}).
We refer to the survey article~\cite{PSU:tensor-survey} for more details on known results.

If $m=1$, $\rt_{H_0}f=0$ and solenoidal injectivity of $\tilde\rt$ imply that $f=\der h$ for some smooth scalar function $h$.
If $m\geq2$, the gauge condition is more complicated.
We will not pursue this direction further here, as it is irrelevant for our main problem.
\end{rmk}

\section{Observability Inequalities}
We prove the main theorem, Theorem \ref{main} in this section. We consider our globally defined complex optics solutions $\tilde{U}_1$ and $\tilde{U}_2$, which were constructed previously. Dropping the superscripts $x_0,\omega_0$ where it is understood, from our approximation we know that the main term of interest is 
\begin{align}\label{hello}
\sabs{|u_1|-|u_2|}=\sabs{|\tilde{U}_1|-|\tilde{U}_2|}+\mathcal{O}(\lambda^{-1/2})
\end{align}
for $x\in \partial M$. 
This estimate is a result of building an approximate solution with $N$ sufficiently large. We conclude from~\eqref{error3} and~\eqref{comperror}
\begin{align}
\norm{\tilde{U}_1-u_1}_{H^s(M')}\leq \frac{C}{\sqrt{\lambda}}. 
\end{align}
Here $C$ is a generic constant, which depends on the $C^{N+s}(M)$ norm of $n_1^2(x)$ with $N$ such that $\max \{(1-d)/2+2+2s,0\}+1$. We use the fact that $\tilde{U}_1$, and $u_1$ are bounded in $C^{N+s}(M)$ norm to conclude
\begin{align}\label{difference2}
\norm{\tilde{U}_1-u_1}_{C^0(M')}\leq \frac{C}{\sqrt{\lambda}} 
\end{align}
by Sobolev embedding, provided $s>d/2$

We use the estimate on the zeroth order terms
\begin{align}\label{difference}
\norm{U_1-a^1\exp(i\lambda\psi)}_{C^0(M')}\leq \frac{C}{\lambda}.
\end{align}
where $a^1$ is the first order term, found to leading order by \eqref{approximation} and \eqref{approximation1}. Combining the estimates~\eqref{difference} and~\eqref{difference2}, gives the estimate~\eqref{hello}. Now we need to combine the estimates to recover the X-ray transform. We would like to use the first order coefficients to reconstruct the coefficient $n^2(x)$. Using~\eqref{approximation}, we know that 
\begin{align}\label{difference3}
&\textstyle{ \sabs{a^1(x)\exp(i\lambda\psi_1(x))-a^1_{0}(0)\exp\sparen{-\int\limits_0^s(n_1^2(x_1(t))\,dt-\int\limits_0^s\mathrm{tr}\mathcal{M}_1(t)\,dt}\exp(i\lambda\psi_1(x))\phi(x)}}\\& \nonumber = \mathcal{O}(\lambda^{-1/2d}),
\end{align}
by choice of the cutoff function in~\eqref{bye}. 

Examining~\eqref{hello}, we are interested in the left hand side. Using~\eqref{difference3}, we would like to approximate it by 
\begin{align}\label{hi}
&\left\lvert a_0(0)\sparen{\exp\sparen{-\int\limits_0^sn_1^2(x_1(t))\,dt}-\exp\sparen{-\int\limits_0^sn_2^2(x_2(t))\,dt}}\right.
\times \\& \nonumber \left.\exp\sparen{-\int\limits_0^s\mathrm{tr}\mathcal{M}_1(t)\,dt}\exp(\lambda\underline{\psi_1(x))}\phi(x)\right\rvert,
\end{align}
where
\begin{align}
\underline{\psi_1(x)}=-(x-x_1(s))\Im \mathcal{M}_1(s)\cdot (x-x_1(s)).
\end{align}

We need the following lemma to control the error, which works for level sets of general Hamiltonians $H$, however we chose to focus on the case $H(x,p)=|p|_g^2-n^2$ which is relevant for our solution curve. Again, we remind the reader the equation \eqref{time3}
is the correct one that we want to solve to high order by the derivation of Lemma \ref{lem:time2}.  Recall Grownwall's identity:

\begin{thm}[Grownwall's Inequality]
Let $w: [t_0,t_1]\rightarrow \mathbb{R}^+$ be continuous and non-negative and suppose that $w$ obeys the integral inequality
\begin{align}
w(t)\leq A+\int\limits_{t_0}^tB(s)w(s)\,ds
\end{align}
for all $t\in[t_0,t_1]$ whenever $A\geq 0$ and $B:[t_0,t_1]\rightarrow\mathbb{R}^+$ is continuous and non-negative. Then we have 
\begin{align}
w(t)\leq A\exp\left(\int\limits_{t_0}^t B(s)\right)\,ds
\end{align}
for all $t\in [t_0,t_1]$. 
\end{thm}

We let $\epsilon\in (0,1)$ be a fixed positive constant. We set $\tilde{H}(x,p)=p\cdot\nabla_{p}H(x,p)$. 
\begin{lem} (Uniqueness)
  \label{uniqueness}
Assume that $\psi_1,\psi_2\in C^\infty(M;\mathbb C)$, $s_0,s\in (a,b)$,
satisfy
$$
\begin{gathered}\label{time2}
\psi_1(x)=\psi_2(x)+\mathcal O((d_{g'}(x,x_1(s_0))^{m+1}) ,\,\textrm{for}\,\, s_0\quad
(\partial_x\psi_j)(x_j(s))=p_j(s),\\
\partial_s\psi_j(x)+\tilde{H}(x,\nabla_g\psi_j(x))=\mathcal O(d_{g'}(x,x_j(s))^{m+1}) 
\end{gathered}
$$
for some $m\in\mathbb N$. Then the following holds
\begin{align}
&\sup\limits_{s\in (a,b)} |x_1(s)-x_2(s)|=\mathcal{O}(\epsilon)\qquad 
\sup\limits_{s\in (a,b)}|\mathcal{M}_1(s)-\mathcal{M}_2(s)|=\mathcal{O}(\epsilon)
\end{align}
\end{lem}
\begin{proof}
We can map the neighborhood of the flow on the manifold to the Euclidean plane and use the metric $g_{ij}=\delta_{ij}$ the standard Euclidean one. In this case we have by definition of the Hamiltonian flow.  
\begin{align}
|\frac{d|\tilde{x}|}{ds}\leq |\frac{d\tilde{x}}{ds}|\leq |(n_1-n_2)(x_1(s))|+C|\tilde{x}|
\end{align}
where $C=\sup\limits_x (\nabla n_2(x))$ by the mean value theorem and $\tilde{x}(s)=x_1(s)-x_2(s)$. 
Given initial conditions $\tilde{x}(0)=0$ we can use Gronwall's inequality with $w=|\tilde{x}|$ and $A=\int \limits_{t_0}^t\epsilon\,ds$  to conclude 
\begin{align}
|\tilde{x}(s)|=\mathcal{O}(\epsilon)
\end{align}
since $\norm{n_1^2-n_2^2}_{C^3(M')}<\epsilon$. The $\mathcal{O}$ terms depend on the length of the geodesic. 

We also have by definition of the matrix Ricatti equations:
\begin{align}
\frac{d|\mathcal{\tilde{M}}(s)|}{ds}\leq |\frac{d}{ds}\mathcal{\tilde{M}}(s)|=-D^2n_1^2(x_1(s))+D^2n_2^2(x_2(s))+\mathcal{M}^2_1(s)-\mathcal{M}^2_2(s)
\end{align}
which then implies
\begin{align}
|\frac{d}{ds}\mathcal{\tilde{M}}(s)|\leq |D^2(n_1^2-n_2^2)(x_1(s))|+C_1|\tilde{x}|+C_2|\tilde{M}|
\end{align}
where $C_1=\sup\limits_x (D^3n_2(x))$ by the mean value theorem and $\tilde{M}(s)=\mathcal{M}_1(s)-\mathcal{M}_2(s)$, $C_2=\sup\limits_{j=1,2,s} \{\mathcal{M}_j(s)\}$. 
Applying Gronwall's theorem again gives the second desired inequality. 
\end{proof}

The expansion implies
\begin{align}\label{taylor}
&(x-x_1(s))=(x-x_2(s))+\mathcal{O}(\epsilon)+\mathcal{O}(d_{g'}(x,x_1(s))^2);\\&
\nonumber
(x-x_1(s))\Im\mathcal{M}_1(s)(x-x_1(s))=\\& \nonumber(x-x_2(s))\Im\mathcal{M}_2(s)(x-x_2(s))+\mathcal{O}(\epsilon)+\mathcal{O}(d_{g'}(x,x_1(s))^3)
\end{align}
and also 
\begin{align}\label{matrix}
\Im\mathcal{M}_1(s)=\Im\mathcal{M}_2(s)+\mathcal{O}(\epsilon)
\end{align}

It follows Lemma~\ref{uniqueness} in local coordinates:
\begin{align}\label{expansion2}
|U_2(x)|=&|a_0(0)\exp\sparen{-\int\limits_0^s(n_2^2(x_2(t))\,dt}
\exp\sparen{-\int\limits_0^s\mathrm{tr}\mathcal{M}_1(t)\,dt}\exp(\lambda\underline{\psi_1(x))}\phi_2(x)|\times \\& \nonumber(\exp(-\mathcal{O}(\lambda\epsilon))
\end{align}
with $\underline{\psi_1(x)}$ as above. Since Lemma~\ref{uniqueness} is local, the order terms depend on diam$(M)$ and $C^3(M)$ norm of $n_1^2(x)$ and $n_2^2(x)$, made precise by \eqref{taylor},\eqref{matrix}.  This allows us to conclude~\eqref{hi}, with appropriate loss of error. 
Notice that even if $\epsilon>\frac{1}{\sqrt{\lambda}}$, if we can re-define $U_2$ with a $\phi_2(x)$ which has larger significantly larger support and still obtain the same estimate as \eqref{expansion2}, since by Lemma \ref{expsize}, the contribution of the error estimates outside of a neighborhood of $\Omega(\lambda^{-1/2})$ is negligible. The cutoffs were kept  with support in $\Omega(\lambda^{-1/2})$ to illustrate that the integration in a neighborhood around the tube is the most important contribution, c.f. \cite{liu}. 

We now examine~\eqref{hi}. Because $\Im \mathcal{M}_1(s)$ is a positive definite matrix, by Corollary~\ref{expsize} we have  
\begin{align}\label{mainsup}
\sup\limits_{x\in M}\sabs{\exp(\lambda\underline{\psi_1(x)})\phi(x)}=1.
\end{align} 
Indeed, $x_1(s)$ reaches the boundary and is contained in $\Omega(\lambda^{-1/2})$. Because the other coefficients of~\eqref{hi} are independent of $x$, we obtain that the supremum of~\eqref{hi} over $x\in\partial M$ equals
\begin{align}
C(s)\sabs{a^1_0(0)\sparen{\exp\sparen{-\int\limits_0^{\tau(x_0,\omega_0)}\alpha(n_1^2(x_1(t))\,dt}-\exp\sparen{-\int\limits_0^{\tau(x_0,\omega_0)}\alpha n_2^2(x_2(t))\,dt}}\times E(x)}.
\end{align}
Here 
\begin{align}
C(s)=\exp\sparen{-\int\limits_0^{\tau(x_0,\omega_0)}\mathrm{tr}\mathcal{M}_1(t)\,dt}.
\end{align}
with $\tau(x_0,\omega_0)$ the time it takes for $x(s)$ defined by \eqref{flow} with initial conditions $(x_0,\omega_0)$ to exit $M$. We also have
\begin{align}
\|E(x)\|_{C^0(M')}=\exp(\mathcal{O}(-\lambda\epsilon))
\end{align} 
by \eqref{expansion2} from Lemma \eqref{uniqueness}. 

We need the following lemma to obtain the ray transform which is taken from~\cite{aw}. 
\begin{lem}\label{lem:smallness}
Let $A(x)$ and $B(x)$ be positive functions in $C^0(\mathbb{R})$ and $\varepsilon\in (0,1)$ such that 
\begin{align}\label{small}
\norm{\exp(-A(x))-\exp(-B(x))}_{C^0(\mathbb{R})}<\varepsilon .
\end{align}
Then there is a constant $C$ depending on the $C^0(\mathbb{R})$ norms of $A$ and $B$, such that
\begin{align}
\norm{A(x)-B(x)}_{C^0(\mathbb{R})}<C\varepsilon.
\end{align}
\end{lem}

\begin{proof}
By the mean value theorem, there exists an $r_*$ between $B(x)$ and $A(x)$ for each fixed $x$ such that 
\begin{align}
\sabs{\sparen{\exp(-A)-\exp(-B)}}=\sabs{\sparen{-\int\limits_B^A\exp(-r)\,dr}}=\sabs{B-A}\exp(-r_*).
\end{align}
The desired result follows by taking the supremum over $x$ and applying~\eqref{small}.  
\end{proof}

Now we can see by assumption that $\delta, \lambda^{-1}$ are less than $\epsilon_0$. Applying Lemma \eqref{lem:smallness} with $\varepsilon=\epsilon_0$ and $I_{H_0}n_1^2=A,\,\, I_{H_0}n_2^2=B$ we obtain: 
\begin{align}\label{hi2}
\norm{I_{H_0}(n_1^2-n_2^2)}_{C^0(\partial\mathcal{S}M^+)}\leq \tilde{C}_1\sparen{\delta+\frac{C_2}{\lambda^{\beta'}}}.
\end{align}
$\tilde{C}_1$ denotes a generic constant depending on the $C^0(M)$ norm of $n_1^2$, $n_2^2$, while $C_2$ depends on the $C^{N+s}(M)$ norm. $H_0$ denotes the flow associated to $n^2_1(x)$. Here we have used the fact that we can Taylor expand $n^2_2(x_2(s))$ around $x_1(s)$, and $x_1(s)-x_2(s)=\mathcal{O}(\epsilon)$ via \eqref{taylor} and the smallness condition~\eqref{eq:smallness}.

Now we set $n_1^2-n_2^2=\tilde{n}^2$, and we recall that it has compact support. Because we have assumed that $M$ is strictly convex, it follows from~\eqref{hi2} that
\begin{align}
\norm{I_{H_0}(n_1^2-n_2^2)}_{C^1(\partial\mathcal{S}M^+)}\leq C_1\sparen{\delta+\frac{C_2}{\lambda^{\beta'}}}^{\frac{1}{2}}.
\end{align}
by the embedding theorems in the Appendix with $k=2,t=1$. Here $C_1$ denotes a generic constant depending on the $C^2(M)$ norm of $n_1^2$, $n_2^2$. The compact support of $\tilde{n}^2$ implies that
\begin{align}\label{hi3}
&\norm{I^*I_{H_0}(n_1^2-n_2^2)}_{C^1(\partial\mathcal{S}M^+)}\leq  C_1\sparen{\delta+\frac{C_2}{\lambda^{\beta'}}}^{\frac{1}{2}}.
\end{align} 
 To finish the proof of Theorem~\ref{main}, we use~\eqref{hi3} along with the stability estimate of corollary~\ref{cor:simple-inj}. In particular, there is a number $c'>0$ that depends on $M$, $g$ and $n_1$ such that the following holds. If we choose initial speeds with $\abs{\omega_0}=c'$, then the flow transform is stably invertible. We notice that $\tilde{n}^2\in C_0^{\infty}(M)$ and also by choice of $n_1^2=n^2_2\equiv 1$ on $\partial M$ the stability estimate over $M'$ makes no difference, so that the Corollary ~\ref{cor:simple-inj} applies with $n^2(x)=-q(x)$ and $H_0=0$. 

\section{Appendix: Embedding theorems}
This result is taken from Section 2.4 in \cite{M}. We let $f_1,f_2\in C^k(M)$ with $k>1$. We assume there is a constant $C>0$ and an $\varepsilon<<1$ such that
\begin{align}
\|f_1\|_{C^k(\overline{M})}+\|f_2\|_{C^k(\overline{M})} \leq C \qquad \norm{f_1-f_2}_{C(\overline{M})}\leq \varepsilon
\end{align}
We use standard interpolation estimates, see Section 4.3.1 in \cite{int} to normalize the use of norm estimates. As an example using interpolation estimates we obtain for generic $f\in C^k(\overline{M})$ with $k>1$ 
\begin{align}
\|f\|_{C^t(\overline{M})}\leq A\|f\|_{C^{t_1}(\overline{M})}^{1-\theta}\|f\|^{\theta}_{C^{t_2}(\overline{M})}
\end{align}
with $0\leq \theta<1$ with $t_1\geq 0, t_2\geq 0$. As an immediate consequence of the assumptions on $f_1,f_2$ we have that
\begin{align}
\|f_1-f_2\|_{C^t(\overline{M})}\leq C\varepsilon^{\frac{(k-t)}{k}}
\end{align}
for each $t\geq 0$ if $k>t$.


\subsection*{Acknowledgements}
J.I.\ was partially supported by an ERC Starting Grant (grant agreement no 307023).
A.W.\ acknowledges support by EPSRC grant EP/L01937X/1 and ERC Advanced Grant MULTIMOD 26718. A. W. thanks Slava Kurylev for interesting mathematical discussions which helped improve the presentation of the article. 

\bibliographystyle{abbrv}
\bibliography{acoustic}

\end{document}